\documentclass[a4paper,12 pt]{amsart}

\usepackage{amsfonts}
\usepackage{amsmath}
\usepackage{amsmath,amsfonts,amssymb,amsthm,latexsym}
\usepackage{graphicx}
\usepackage{hyperref}

\newcommand{\RR}{\mathbb{R}}
\newcommand{\CC}{\mathbb{C}}
\newcommand{\NN}{\mathbb{N}}

\newtheorem{Tw}{Theorem}
\newtheorem{Le}{Lemma}
\newtheorem{Stw}{Proposition}

\newtheorem{Wn}{Corollary}

\theoremstyle{remark}
\newtheorem{Uw}{Remark}

\theoremstyle{definition}
\newtheorem{Df}{Definition}

\begin{document} 
\title[The polyharmonic Bergman space]{The polyharmonic Bergman space for the union of rotated unit balls}
\author{Hubert Grzebu{\l}a}
\address{Faculty of Mathematics and Natural Sciences,
College of Science\\
Cardinal Stefan Wyszy\'nski University\\
W\'oycickiego 1/3, 
01-938 Warszawa, Poland}
\email{h.grzebula@student.uksw.edu.pl}

\begin{abstract}
In the paper we consider the polyharmonic Bergman space for the union of the rotated unit Euclidean balls. Using so called zonal polyharmonics we derive the formulas for the kernel of this space. Moreover, we study the weighted polyharmonic Bergman space. By the same argument we get the Bergman kernel for this space.
 \end{abstract}
\maketitle   
\section{Introduction}
The polyharmonic Bergman spaces have recently been extensively studied (see \cite{P}, \cite{P2}, \cite{T} or \cite{T2}). They are mainly considered on the unit ball  or on its complement. However, we regard the space of polyharmonic and square integrable functions on the set $\widehat{B}_p=\bigcup_{k=0}^{p-1}e^\frac{k\pi i}{p}B$ (in fact, we may assume the polyharmonicity only on $B$, because every polyharmonic function can be extended analytically from $B$ onto every rotated Euclidean ball). More precisely, we consider the space of the  polyharmonic functions on $\widehat{B}_p$ such that
\begin{equation*}
||u||_{b_p^2}:=\left(\frac{1}{p}\sum_{k=0}^{p-1}\int\limits_B\left|u(e^\frac{k\pi i}{p}y) \right|_{\CC}^2dy \right)^{1/2}
<\infty.
\end{equation*} 

Such space we denote by $b_p^2(\widehat{B}_p)$. The motivation to study the polyharmonic functions on such set is given in the paper \cite{G-M} (see also \cite{L} and \cite{S-M}). 

Using the mean value property for polyharmonic functions (see Lemma 5 in \cite{P}) and some theorems, we conclude that $b_p^2(\widehat{B}_p)$ is a Hilbert space with the inner product
$$\langle u,v \rangle_{b^2_p}=\frac{1}{p }\sum_{k=0}^{p-1}\int\limits_{B}u(e^\frac{k\pi i}{p}y)\overline{v(e^\frac{k\pi i}{p}y)}\,dy.$$
Further, by theorem of Riesz, there exists a function $R_p(x,\cdot)\in b_p^2(\widehat{B}_p)$ such that $u(x)=\langle u,R_p(x,\cdot) \rangle_{b^2_p}$ for every $u\in b_p^2(\widehat{B}_p)$. The function $R_p(x,\cdot)$ is a reproducing kernel for the Bergman space and it is called the Bergman kernel for $\widehat{B}_p$. Using some properties of spherical polyharmonics and zonal polyharmonics we get the formula  for the Bergman kernel which is similar to the harmonic Bergman kernel:
$$R_p(x,y)=\frac{1}{n\Omega_n}\sum_{m=0}^{\infty}(n+2m)Z^p_m(x,y),$$
where $\Omega_n=\pi^{n/2}/\Gamma(n/2+1)$ is the volume of the unit ball $B$ in $\RR^n$.
By the formula for polyharmonic Poisson kernel (see Theorem 4 in  \cite{G-M-2}) we can express the Bergman kernel in the term of polyharmonic Poisson kernel $P_p(x,y)$
\begin{eqnarray*}
R_p(x,y) &=& \frac{1}{n\Omega_n}\left( nP_p(x,y)+\frac{d}{dt}P_p(tx,ty)\biggr|_{t=1}\right)
\end{eqnarray*}
and from this we obtain the explicit formula for Bergman kernel
$$ R_p(x,y) = \frac{(n-4p)|x|^{2p+2}|\overline{y}|^{2p+2}+(8px\overline{y}-n-4p)|x|^{2p}|y|^{2p}+n(1-|x|^{2}|\overline{y}|^{2})}{n\Omega_n(1-2x\overline{y}+|x|^2|\overline{y}|^2)^{n/2+1}}. $$
Moreover, we can express $R_p(x,y)$ in the terms of the harmonic Bergman kernel $R(x,y)$ and the harmonic Poisson kernel $P(x,y)$:
$$R_p(x,y)=\frac{1-|x|^{2p}|\overline{y}|^{2p}}{1-|x|^2|\overline{y}|^{2}} R(x,y)+\frac{1}{n\Omega_n}\sum_{k=0}^{p-1}4k|x|^{2k}|\overline{y}|^{2k}P(x,y).$$
Next we consider the weighted polyharmonic Bergman space. Here we study polyharmonic functions on $\widehat{B}_p$, which satisfy the following condition
\begin{equation*}
||u||_{b_{p,\alpha,\beta}^2}:=\left(\frac{1}{p}\sum_{k=0}^{p-1}\int\limits_B\left|u(e^\frac{k\pi i}{p}y) \right|_\CC^2|y|^\alpha(1-|y|^{2})^\beta dy \right)^{1/2}
<\infty,
\end{equation*}
where $n+\alpha>0, \beta>-1$. We denote this space by $b_{p,\alpha,\beta}^2(\widehat{B}_p)$. By the similar arguments we prove that $b_{p,\alpha,\beta}^2(\widehat{B}_p)$ is a Hilbert space and there exists the reproducing kernel $R_{p,\alpha,\beta}(x,\cdot)\in b_{p,\alpha,\beta}^2(\widehat{B}_p)$ such that
$$u(x)=\frac{1}{p}\sum_{k=0}^{p-1}\int\limits_B 
 u(e^\frac{k\pi i}{p}y) \overline{R_{p,\alpha,\beta}(x,e^\frac{k\pi i}{p}y)} |y|^\alpha(1-|y|^{2})^\beta dy$$
for every $u\in b_{p,\alpha,\beta}^2(\widehat{B}_p)$.
The function $R_{p,\alpha,\beta}$ is called a polyharmonic weighted Bergman kernel. We get the formula for this kernel
$$R_{p,\alpha,\beta}(x,y)=\frac{1}{n\Omega_n}\sum_{m=0}^{\infty}\frac{2\Gamma(m+\beta+1+\frac{n+\alpha}{2})}{ \Gamma(\beta+1)\Gamma(m+\frac{\alpha+n}{2})}  Z^p_m(x,y).$$
Moreover
$$R_{p,\alpha,\beta}(x,y)=\sum_{k=0}^{p-1}|x|^{2k}|\overline{y}|^{2k}R_{1,\alpha+4k,\beta}(x,y).$$

The paper is organised as follows. In the next section we give some basic notations and one lemma about extension of polyharmonic functions from the real ball onto its rotation (Lemma \ref{L1}).

In the third section we recall some informations about the spherical polyharmonics,  zonal polyharmonics,  polyharmonic Poisson kernel and their properties (Lemmas \ref{L2}-\ref{L6}). By these lemmas we get another properties for polyharmonic functions (Proposition \ref{S1}, Proposition \ref{S2}).

In the next section we introduce the polyharmonic Bergman space. Using Lemma \ref{L7} we get some properties for this space (Proposition \ref{S3} and Proposition \ref{S4}).

In the fifth section we introduce the polyharmonic Bergman kernel for the set $\widehat{B}_p$. We give  basic properties for this function (Proposition \ref{S5}) and some another properties for the polyharmonic functions (Propositions \ref{S6}--\ref{S8}). Using these properties and those ones given in Section 3 we get the formula for the polyharmonic Bergman kernel (Theorem \ref{T1}). Moreover we get the explicit form for this function (Theorem \ref{T2}) and we express it in the terms of harmonic Bergman kernel and harmonic Poisson kernel (Theorem \ref{T3}).

In the sixth section we consider the polyharmonic weighted Bergman space. Similarly as for unweighted one we show that
this space is a Hilbert space and there exists the reproducing kernel called the weighted Bergman kernel (Corollary \ref{W1}). By similar arguments (Lemma \ref{L9}, Propositions \ref{S9} and \ref{S10}),  we get the formulas for this kernel (Theorems \ref{T4}--\ref{T6}).

\section{Preliminaries}
In this section we give some basic notations and definitions.

We define the real norm
\begin{equation*}
|x|= ( \sum_{j=1}^nx_j^2 )^{1/2}\quad\textrm{for}\quad x=(x_1,\dots,x_n)\in \RR^n
\end{equation*}
and the complex norm
\begin{equation*}||z||= (\sum_{j=1}^n|z_j|_{\CC}^2 )^{1/2} \quad\textrm{for}\quad z=(z_1,\dots,z_n)\in \CC^n
\end{equation*} 
with $ |z_j|_{\CC}^2=z_j \overline{z}_j$. We will also use the complex extension of the real norm for complex vectors: 
\begin{gather*}
|z|=( \sum_{j=1}^nz_j^2 )^{1/2}\quad\textrm{for}\quad z=(z_1,\dots,z_n)\in\CC^n.
\end{gather*}
By 
$$xy=x_1y_1+x_2y_2+\dots +x_ny_n$$
we denote the usual inner product for the complex (real) vectors $x,y$.

By a square root in the above formula we mean the principal square root, where a branch cut is taken along the non-positive real axis.
Obviously the function $|\cdot|$ is not a norm in $\CC^n$, because it is complex valued and hence the function $|z-w|$ is not a metric on $\CC^n$.

We will  consider mainly complex vectors of the form $ z=e^{i\varphi}x $, that is vectors $x\in \RR^n $ rotated in $\CC^n$ by the angle $\varphi$.

For the set $G\subseteq\RR^n$ and the angle $\varphi\in\RR$ we will consider the rotated set defined by
$$e^{i\varphi}G:=\{e^{i\varphi}x:x\in G\}.$$
We will consider mainly the following unions of rotated sets in $\CC^n$:
$$
\widehat{B}_p:=\bigcup_{k=0}^{p-1}e^{\frac{k\pi i}{p}}B\quad\textrm{and}\quad\widehat{S}_p:=\bigcup_{k=0}^{p-1}e^{\frac{k\pi i}{p}}S\quad\textrm{for}
\quad p\in\NN,
$$
where $B$ and $S$ are respectively the unit ball and sphere in $\RR^n$ with a centre at the origin.

Let $G$  be an open set in $ \RR^n $. We denote by $\mathcal{A}(G)$ the space of analytic functions on $ G $.
Similarly we say that $f\in \mathcal{A}(e^{i\varphi}G)$ if and only if $f_{\varphi}(x):=f(e^{i\varphi}x)\in \mathcal{A}(G)$. We call
$$\mathcal{A}_{\Delta}(G):=\{f\in \mathcal{A}(G):\Delta_x f=0 \}$$ 
the space of harmonic functions on $ G $, where $\Delta_x $ denotes the Laplacian in $\RR^n$.
Analogously we define the family $\mathcal{A}_{\Delta}(e^{i\varphi}G)$ of harmonic functions on $e^{i\varphi}G$. Observe that
$f\in \mathcal{A}_{\Delta}(e^{i\varphi}G)$ if and only if $f_{\varphi}\in\mathcal{A}_{\Delta}(G)$. 
Similarly, replacing the Laplace operator $\Delta_x$ by its $p$-th iteration $\Delta_x^p$ in the above definitions,
we introduce the spaces of polyharmonic functions of degree $p$, that is $ \mathcal{A}_{\Delta^p}(G)$
and $\mathcal{A}_{\Delta^p}(e^{i\varphi}G)$.

We will use the following lemma (see also Siciak's Theorem, \cite{S})

\begin{Le}[{\cite[Lemma 1]{G-M}}]
\label{L1}
Let $\varphi\in\RR$ and $u\in \mathcal{A}_{\Delta^p}(B)$. Then 
the function $u$ has a holomorphic extension to the set $\{z\in\CC^n\colon\ z=e^{i\psi}x,\ \psi\in\RR,\ x\in B\}$,
whose restriction $u_{\varphi} $ to $e^{i\varphi}B$ is polyharmonic of order $p$, i.e. $u_{\varphi}\in \mathcal{A}_{\Delta^p}(e^{i\varphi}B)$.
\end{Le}

\section{Zonal polyharmonics and polyharmonic Poisson kernel}

In this section we recall the spherical and zonal polyharmonics and the polyharmonic Poisson kernel (see \cite{G-M-2}).

Let $m,p\in \mathbb{\NN}$. We denote by $ \mathcal{H}_m^p(\CC^n)$ the space of polynomials on $\CC^n$, which are homogeneous of degree $m$
and are polyharmonic of order~$p$. By homogeneous polynomial of degree $m$ we mean the polynomial $q$ such that
$$q(az)=a^mq(z)\quad\textrm{for every}\quad a\in \CC.$$

Let's observe that if $m<2p,$ then $ \mathcal{H}_m^p(\CC^n)$ is the same as the space of homogeneous polynomials of degree $m$.
\begin{Df}[{\cite[Definition 1]{G-M-2}}]
\label{D1}
The restriction to the set $\widehat{S}_p:= \bigcup_{k=0}^{p-1}e^{\frac{k\pi i}{p}}S $ of an element of $\mathcal{H}_m^p(\CC^n)$
is called a \emph{spherical polyharmonic of degree $m$ and order $p$}.

The set of spherical polyharmonics is denoted by $ \mathcal{H}_m^p( \widehat{S}_p ) $, so
\begin{equation*}
\mathcal{H}_m^p( \widehat{S}_p ):=\left\{ u|_{\widehat{S}_p}\colon u\in \mathcal{H}_m^p(\CC^n)   \right\}.
\end{equation*}
\end{Df}

The spherical polyharmonics of order $1$ are called \emph{spherical harmonics} and their space is denoted by $\mathcal{H}_m(S):=\mathcal{H}_m^1(S)$
(see \cite[Chapter 5]{A-B-R}). Analogously we write $\mathcal{H}_{m}(\CC^n)$ instead of $\mathcal{H}_{m}^1(\CC^n)$.

We shall recall some properties of spherical polyharmonics. To do this let us consider the Hilbert space $ L^2 ( \widehat{S}_p )  $ of square-integrable functions on $\widehat{S}_p$ with the inner product defined by
\begin{equation}
\label{e1}
\left\langle f,g\right\rangle _{ \widehat{S}_p }:=\frac{1}{p} \int\limits_S \sum_{j=0}^{p-1}f(e^{\frac{j\pi i}{p}} \zeta )
\overline{g(e^{\frac{j\pi i}{p}} \zeta )}\,d\sigma (\zeta),
\end{equation}
where $d\sigma$ is a normalized surface-area measure on the unit sphere $S$.
\begin{Le}[{\cite[Propositions 4]{G-M-2}}]
\label{L2}
The space $\mathcal{H}^p_m(\widehat{S}_p)$ is finite dimensional.
\end{Le}
\begin{Le}[{\cite[Theorem 1]{G-M-2}}]
\label{L3}
The space $L^2 ( \widehat{S}_p )$ is the \emph{direct sum} of spaces $ \mathcal{H}^p_m(\widehat{S}_p) $ and we write 
\begin{equation*}
L^2 ( \widehat{S}_p )=\bigoplus_{m=0}^{\infty} \mathcal{H}_m^p ( \widehat{S}_p ).
\end{equation*}
It means that
\begin{enumerate}
\item[(i)] $\mathcal{H}^p_m(\widehat{S}_p)$ is a closed subspace of $L^2 ( \widehat{S}_p )$ for every $m$.
\item[(ii)] $\mathcal{H}^p_m(\widehat{S}_p)$ is orthogonal to $\mathcal{H}^p_k(\widehat{S}_p)$ if $m\neq k$.
\item[(iii)] For every $ x \in L^2 ( \widehat{S}_p )$ there exist $ x_m \in \mathcal{H}^p_m(\widehat{S}_p)$  such that $ x=x_0+x_1+x_2+\dots$, where the sum is converging in the norm of $L^2 ( \widehat{S}_p )$.
\end{enumerate}
\end{Le}

By Lemma \ref{L3} we may consider $\mathcal{H}_m^p (\widehat{S}_p)$ as a Hilbert space  with the inner product (\ref{e1}) induced
from $L^2( \widehat{S}_p )$.
 
Let $ \eta \in \widehat{S}_p$ be a fixed point. Let us consider the linear functional
$\Lambda_{\eta}\colon \mathcal{H}_m^p (\widehat{S}_p) \rightarrow {\CC} $ defined as 
\begin{gather*}
\Lambda_{\eta} (q)=q(\eta)\quad\textrm{for}\quad q\in \mathcal{H}_m^p (\widehat{S}_p).
\end{gather*}
Since $\mathcal{H}_m^p (\widehat{S}_p)$ is a finite dimensional inner-product space, it is a self-dual Hilbert space so there exists a unique $Z_m^p(\cdot,\eta)\in \mathcal{H}_m^p (\widehat{S}_p)$ such that 
\begin{equation}
\label{e2}
q(\eta)=\left\langle q,Z^p_m(\cdot,\eta)\right\rangle_{  \widehat{S}_p }\quad\textrm{for every}\quad q\in
\mathcal{H}^p_m(\widehat{S}_p).
\end{equation}

\begin{Df} [{\cite[Definition 2]{G-M-2}}]
\label{D2}
The function $Z^p_m(\cdot,\eta)$ satisfying (\ref{e2}) is called a \emph{zonal polyharmonic}
of degree $m$ and of order $p$ with a pole $\eta$. 
\end{Df}

Zonal polyharmonics of order $p=1$ are called \emph{zonal harmonics}. Throughout this paper we will denote them by
$Z_m(\cdot,\eta)$ instead of $Z^1_m(\cdot,\eta)$ for $\eta \in S$. 

Let's observe that we can extend the definition of zonal harmonics from $S\times S$ on $\widehat{S}_p\times \widehat{S}_p$ as follows:
$$Z_m(e^\frac{j\pi i}{p}\zeta,e^\frac{l\pi i}{p}\eta):=e^\frac{m(j-l)\pi i}{p}Z_m(\zeta,\eta)$$
for any $\zeta,\eta\in S$ and $j,l=0,1,\dots,p-1.$ Moreover, the zonal harmonics are extended on $B\times B$ (see 8.7 in \cite{A-B-R}) and hence, by Lemma \ref{L1}, they are extended on $\widehat{B}_p\times\widehat{B}_p$.

Let's give some properties of the zonal polyharmonics.    
\begin{Le}[{\cite[Theorem 2]{G-M-2}}]
\label{L4}
Let $\zeta,\eta \in \widehat{S}_p$, then
\begin{equation*}
Z^p_m(\zeta,\eta)=\sum_{k=0}^{p-1} |\zeta|^{2k}|\overline{\eta}| ^{2k}Z_{m-2k}(\zeta,\eta),
\end{equation*}
where in the case $m<2p$ we have  $Z_{m-2k}(\zeta,\eta)\equiv 0$ for $m<2k$ (see Remark 6 in \cite{G-M-2}).
\end{Le}

We may extend the zonal polyharmonics  from $\widehat{S}_p\times \widehat{S}_p $ on $\widehat{B}_p\times \widehat{B}_p$ in the same way as for the zonal harmonics, therefore by Lemma \ref{L4} we have
\begin{equation}
\label{e3}
 Z_m^p(x ,y)=\sum_{k=0}^{p-1}|x|^{2k}|\overline{y}|^{2k}Z_{m-2k}(x,y).
\end{equation}
 In particular
\begin{gather}
\label{e4}
Z^p_m(x,\eta)=\sum_{k=0}^{p-1} |x|^{2k}|\overline{\eta}| ^{2k}Z_{m-2k}(x,\eta)\quad\textrm{for}\quad x\in \widehat{B}_p,
\end{gather}
so by (\ref{e2}) and homogeneity we may write
\begin{gather}
\label{e5}
q(x)=\int\limits_S q(\zeta)Z^p_m(x,\zeta)d\sigma(\zeta)\quad\textrm{for}\quad x\in \widehat{B}_p.
\end{gather}

 The zonal polyharmonics give us the construction of the polyharmonic Poisson kernel. 

\begin{Df}[{\cite[Definition 6]{G-M-2}}]
\label{D3}
The function $P_p\colon(\widehat{B}_p\times \widehat{S}_p) \cup(\widehat{S}_p \times \widehat{B}_p)\to\CC$ is called a \emph{Poisson kernel
for $\widehat{B}_p$} provided for every polyharmonic function $u$ on $\widehat{B}_p$ which is continuous on $\widehat{B}_p\cup\widehat{S}_p$ and for each
$x\in \widehat{B}_p$ holds
\begin{gather*}
u(x)=\left\langle u, P_p(\cdot,x) \right\rangle_{\widehat{S}_p}=\frac{1}{p}\sum_{j=0}^{p-1}\int\limits_S u(e^{\frac{j\pi i}{p}}\zeta)
\overline{P_p(e^{\frac{j\pi i}{p}}\zeta ,x )}\, d\sigma(\zeta).
\end{gather*}
\end{Df}

When $p=1 $, the function $P(x,\zeta):=P_1(x,\zeta)$ is the classical Poisson kernel for the real ball (see \cite[Proposition 5.31]{A-B-R} and \cite[Theorem 3]{F-M})
\begin{equation*}
P(x,\zeta)=\sum_{m=0}^{\infty}Z_m(x,\zeta) = \frac{1-|x|^2|\overline{\zeta}|^2}{(|x|^2|\overline{\zeta}|^2-2x\overline{\zeta}+1)^{n/2}}.
\end{equation*}

\begin{Le}[{\cite[Theorem 4]{G-M-2}}]
\label{L5}
The Poisson kernel has the expansion
\begin{equation}
\label{e6}
P_p(x,\zeta)=\sum_{m=0}^{\infty}Z^p_m(x,\zeta)=\frac{1-|x|^{2p}|\overline{\zeta}|^{2p}}{(|x|^2|\overline{\zeta}|^2-2x \overline{\zeta}+1)^{n/2}}\quad\textrm{for}
\quad x\in \widehat{B}_p,\,\, \zeta \in \widehat{S}_p.
\end{equation}
The series converges absolutely and uniformly on $K\times \widehat{S}_p$, where $K$ is a compact subset of $\widehat{B}_p$.
\end{Le}
\begin{Le}[{\cite[Corollary 1]{G-M}}]
\label{L6}
If $u$ is polyharmonic of order $p$ on $B(a,r)$ and continuous on the set $a+\bigcup_{k=0}^{p-1}e^\frac{k\pi i}{p}\overline{B(0,r)}$, then 
\begin{equation}
\label{e7}
 u(x)=\frac{1}{p}\sum_{k=0}^{p-1}\int\limits_{S}
 \frac{ r^{2p}-|x-a|^{2p}}{r^{2p-n} \left| e^\frac{-k\pi i}{p}(x-a)-r\zeta \right|^n }u \left(a+re^\frac{k\pi i}{p}\zeta \right)d\sigma(\zeta).
\end{equation}
\end{Le}
\begin{Stw}
\label{S1}
Let $(u_n)$ be a sequence of polyharmonic functions of order $p$ such that $u_n\rightrightarrows u$ on every compact subset of $\widehat{B}_p$. Then u is polyharmonic on $\widehat{B}_p$. 
\end{Stw}
\begin{proof}
Let $(u_n)$ be the sequence as in Proposition \ref{S1}. Suppose that $\overline{B}(a,r)\subset B$, then $u_n$ is polyharmonic on $a+\bigcup _{k=0}^{p-1}e^\frac{i k\pi }{p}B$ and continuous on its closure. By Lemma \ref{L6} we may write
$$u_n(x)=\frac{1}{p}\sum_{k=0}^{p-1}\int\limits_S\frac{r^{2p}-|x-a|^{2p}}{r^{n-2p}|e^\frac{-i k\pi}{p}(x-a)-\zeta|^n}u_n(a+r e^\frac{ik\pi}{p}\zeta)d\sigma(\zeta).$$
Since the Poisson kernel for $\widehat{B}_p$ is polyharmonic so continuous, there exists constant $M>0$ such that 
$$\left|\frac{r^{2p}-|x-a|^{2p}}{r^{n-2p}|e^\frac{-i k\pi}{p}(x-a)-\zeta|^n}\right|_\CC\leq M$$
for every $x\in K$, where $K$ is a compact subset of $\widehat{B}_p$. Therefore
\begin{eqnarray*}
\left|u_n(x)-\frac{1}{p}\sum_{k=0}^{p-1}\int\limits_S\frac{r^{2p}-|x-a|^{2p}}{r^{n-2p}|e^\frac{-i k\pi}{p}(x-a)-\zeta|^n}u(a+r e^\frac{ik\pi}{p}\zeta)d\sigma(\zeta)\right|_{\CC}\\
\leq \frac{M}{p}\sum_{k=0}^{p-1}\int\limits_S \left|u_n(a+r e^\frac{ik\pi}{p}\zeta)-u(a+r e^\frac{ik\pi}{p}\zeta)\right|_\CC d\sigma(\zeta).
\end{eqnarray*}
Since $u_n  \rightrightarrows u$ on every compact subset of $\widehat{B}_p$, we can take the limit under the integral sign. Hence we conclude that 
$$u(x)=\frac{1}{p}\sum_{k=0}^{p-1}\int\limits_S\frac{r^{2p}-|x-a|^{2p}}{r^{n-2p}|e^\frac{-i k\pi}{p}(x-a)-\zeta|^n}u(a+r e^\frac{ik\pi}{p}\zeta)d\sigma(\zeta)$$
so $u$ is polyharmonic on $\widehat{B}_p$.
\end{proof}

\begin{Stw}
\label{S2}
Let $u$ be polyharmonic on $\widehat{B}_p$, then there exist $q_m \in \mathcal{H}_m^p(\CC^n)$ such that
$$u(x)=\sum_{m=0}^{\infty}q_m(x)$$
for $x\in \widehat{B}_p$. The series converges absolutely and uniformly on compact subsets of $\widehat{B}_p$.
\end{Stw}
\begin{proof}
The proof is almost the same as in the harmonic case (see {\cite[Corollary 5.34]{A-B-R}}).
\end{proof}

At the end let us observe that by (\ref{e3}) and (\ref{e6}) we may extend the polyharmonic Poisson kernel:

\begin{multline*}
P_p(x,y):=\sum_{m=0}^{\infty}Z^p_m(x,y)=\sum_{m=0}^{\infty}Z^p_m\left( x|\overline{y}|,\frac{y}{|y|}\right)\\
=P_p\left( x|\overline{y}|,\frac{y}{|y|}\right)=\frac{1-|x|\overline{y}||^{2p}}{(1-2x|\overline{y}|\overline{\frac{y}{|\overline{y}|}}+|x|\overline{y}||^2|\overline{\frac{y}{|\overline{y}|}}|^2)^{n/2}}
\end{multline*}
thus
\begin{equation}
\label{e8}
P_p(x,y)=\sum_{m=0}^{\infty}Z^p_m(x,y)=\frac{1-|x|^{2p}|\overline{y}|^{2p}}{(1-2x\overline{y}+|x|^2|\overline{y}|^2)^{n/2}}.
\end{equation}
Let's note that by the above considerations and the Lemma \ref{L1}  we can also extend the harmonic Poisson kernel $P(x,y)$ onto $\widehat{B}_p\times \widehat{B}_p$.

\section{The Bergman space}

Let us consider the polyharmonic functions of order $p$ on  $\widehat{B}_p$ such that
\begin{equation*}
||u||_{b_p^2}:=\left(\frac{1}{p}\sum_{k=0}^{p-1}\int\limits_B\left|u(e^\frac{k\pi i}{p}y) \right|_\CC^2dy \right)^{1/2}
<\infty.
\end{equation*} 
The above condition makes sense because every polyharmonic function can be extended onto any rotated ball by Lemma \ref{L1}. The set of polyharmonic  functions that satisfy the above condition is called the polyharmonic Bergman space, abbreviated the Bergman space, and we denote it by $b_p^2(\widehat{B}_p)$, hence
$$b_p^2(\widehat{B}_p):=\mathcal{A}_{\Delta^p}(B)\cap L^2(\widehat{B}_p).$$ 

When $p=1$, we have the classical harmonic Bergman space $b^2(B)$ (see Chapter 8 in \cite{A-B-R}).
As in the harmonic case, we want to show that $b_p^2(\widehat{B}_p)$ is also a Hilbert space. To do this we will use the following mean value property for polyharmonic functions:
\begin{Le}[Mean value property, {\cite[Lemma 5]{P}}] 
\label{L7}
For every compact subset $K\subset B$ and $x\in K$,
there exists a constant $C=C(K,n,p)$ such that
$$|u(x)|_\CC^2\leq C \int\limits_B |u(y)|_\CC^2 dy$$
for every $u\in \mathcal{A}_{\Delta^p}(B)$.
\end{Le}    

\begin{Stw}
\label{S3}
For every compact subset $K\subset \widehat{B}_p$ and $x\in K$,
there exists a constant $C=C(K,n,p)$ such that
$$|u(x)|_{\CC}\leq C ||u||_{b_p^2}$$ 
for every $u\in b_p^2(\widehat{B}_p)$.
\end{Stw}
\begin{proof}
Let $K\subset B$ be compact, $x\in K$ and $j=0,1,\dots,p-1$. By Lemma \ref{L1} and Lemma \ref{L7}, there exists a positive constant $C=C(K,n,p)$ such that
$$|u(e^{\frac{j\pi i}{p}}x)|_\CC^2\leq C\int\limits_B |u(e^{\frac{j\pi i}{p}}y)|_\CC^2dy.$$
In particular
$$|u(e^{\frac{j\pi i}{p}}x)|_\CC^2\leq C\int\limits_B |u(e^{\frac{j\pi i}{p}}y)|_\CC^2dy+C\sum_{\substack{k=0\\k\neq j}}^{p-1}\int\limits_B |u(e^{\frac{k\pi i}{p}}y)|_\CC^2dy.$$
Hence for every compact $K\subset \widehat{B}_p$ and $x\in K$ we have
$$|u(x)|_\CC^2\leq C \sum_{k=0}^{p-1}\int\limits_B |u(e^{\frac{k\pi i}{p}}y)|_\CC^2dy=Cp||u||^2_{b^2_p}.$$
\end{proof}

\begin{Stw}
\label{S4}
The Bergman space $b_p^2(\widehat{B}_p)$ is a closed subspace of the Hilbert space  $L^2(\widehat{B}_p)$ with the inner product:
\begin{equation}
\label{e9}
\langle u,v \rangle_{b^2_p}=\frac{1}{p }\sum_{k=0}^{p-1}\int\limits_{B}u(e^\frac{k\pi i}{p}y)\overline{v(e^\frac{k\pi i}{p}y)}\,dy.
\end{equation}
\end{Stw}
\begin{proof}
The proof follows from Proposition \ref{S3}, it is analogous as in the  harmonic case (see {\cite[Corollary  8.3]{A-B-R}}).
\end{proof}

\section{The Bergman kernel}

Let $x\in\widehat{B}_p$ be a fixed point. Let's consider the linear functional $\Lambda_x:b_p^2(\widehat{B}_p)\rightarrow \CC$ such that $\Lambda_x(u)=u(x)$. By Proposition \ref{S3}, $\Lambda_x$ is bounded. Since $b_p^2(\widehat{B}_p)$ is a Hilbert space with the inner product (\ref{e9}), by Riesz theorem, there exists a unique function $R_p(x,\cdot)\in b_p^2(\widehat{B}_p)$ such that
\begin{equation}
\label{e10}
u(x)=\langle u,R_p(x,\cdot) \rangle_{b^2_p}=\frac{1}{p }\sum_{k=0}^{p-1}\int\limits_{B}u(e^\frac{k\pi i}{p}y)\overline{R_p(x,e^\frac{k\pi i}{p}y)}\,dy
\end{equation}
for every  $u\in b_p^2(\widehat{B}_p)$. 
The function  $R_p(x,\cdot)$ is called \emph{the polyharmonic Bergman kernel }  for $\widehat{B}_p$.

The function $R(x,\cdot):=R_1(x,\cdot)$ is the harmonic Bergman kernel and this function is given by (see Theorem 8.9 in \cite{A-B-R}):
$$R(x,y)=\frac{1}{n\Omega_n}\sum_{m=0}^{\infty}(n+2m)Z_m(x,y) \quad\textrm{for}\quad x,y\in B.$$ 
Using Lemma \ref{L1}, we extend $R(x,y)$ on $\widehat{B}_p\times \widehat{B}_p$ (we can also use the fact that the functions $Z_m(x,y)$ are extended on $\widehat{B}_p\times \widehat{B}_p$, see section 3).
\begin{Stw}
\label{S5}
The Bergman kernel has the following properties:\\
(1) $R_p(x,y)=\overline{R_p(y,x)}$,\\
(2) $||R_p(x,\cdot)||_{b_p^2}^2=R_p(x,x)$ for $x\in \widehat{B}_p$,\\
(3) the map $Q:L^2(\widehat{B}_p)\rightarrow b_p^2(\widehat{B}_p)$ such that 
$$Q[u](x)=\frac{1}{p}\sum_{k=0}^{p-1}\int\limits_{B}u(e^\frac{k\pi i}{p}y)\overline{R_p(x,e^\frac{k\pi i}{p}y)}\,dy \ \ \ \ for \ \ u\in L^2(\widehat{B}_p), \ \ x\in \widehat{B}_p$$
is a unique orthogonal projection of $L^2(\widehat{B}_p)$ onto $b_p^2(\widehat{B}_p)$.
\end{Stw}
\begin{proof}
The proof is almost the same as in the harmonic case (see {\cite[Proposition 8.4]{A-B-R}}).
\end{proof}
\begin{Stw}
\label{S6}
Let $m\neq l$ and $u\in\mathcal{H}^m_p(\CC^n)$, $v\in\mathcal{H}^l_p(\CC^n)$, then $\langle u,v \rangle_{b^2_p}$=0.
\end{Stw}
\begin{proof}
\begin{eqnarray*}
\langle u,v \rangle_{b^2_p}&=&\frac{1}{p }\sum_{k=0}^{p-1}\int\limits_{B}u(e^\frac{k\pi i}{p}y)\overline{v(e^\frac{k\pi i}{p}y)}\,dy\\
&=&\frac{n\Omega_n}{p }\sum_{k=0}^{p-1}\int\limits_0^1 r^{n-1}\int\limits_{S}u(e^\frac{k\pi i}{p}r\zeta)\overline{v(e^\frac{k\pi i}{p}r\zeta)}\,d\sigma(\zeta)dr\\
&=&\frac{n\Omega_n}{p }\sum_{k=0}^{p-1}\int\limits_0^1 r^{n-1+m+l}\int\limits_{S}u(e^\frac{k\pi i}{p}\zeta)\overline{v(e^\frac{k\pi i}{p}\zeta)}\,d\sigma(\zeta)dr=0,
\end{eqnarray*}
where in the last equality we use Lemma \ref{L3}.
\end{proof}
\begin{Stw}
\label{S7}
Let $u$ be a polynomial of degree $M$, then
$$u(x)=\frac{1}{pn\Omega_n}\sum_{m=0}^{M}(n+2m)\sum_{k=0}^{p-1}\int\limits_{B}u(e^{\frac{k\pi i}{p}}y)Z^p_m(x,e^{\frac{k\pi i}{p}}y)dy \quad\textrm{for}\quad x\in\widehat{B}_p.$$
\end{Stw}
\begin{proof}
Suppose first that $u\in\mathcal{H}^p_m(\CC^n)$. Then 
$$u(x)=\int\limits_{S}u(\zeta)Z^p_m(x,\zeta)d\sigma(\zeta)$$
for every $x\in\widehat{B}_p$.
We have
\begin{eqnarray*}
\int\limits_{B}u(y)Z^p_m(x,y)dy&=&n\Omega_n\int\limits_{0}^1r^{n-1}\int\limits_S u(r\zeta)Z^p_m(x,r\zeta)d\sigma(\zeta) dr\\
&=&n\Omega_n\int\limits_0^1r^{2m+n-1}\left[ \int\limits_S u(\zeta)Z^p_m(x,\zeta)d\sigma(\zeta)\right]dr\\
&=&n\Omega_n u(x)\int\limits_0^1 r^{2m+n-1}dr=\frac{n\Omega_n}{n+2m}u(x).
\end{eqnarray*}
Hence
$$u(x)=\frac{n+2m}{n\Omega_n }\int\limits\limits_{B}u(y)Z^p_m(x,y)dy.$$ 
By the homogeneity we may write
$$u(x)=\frac{n+2m}{pn\Omega_n }\sum_{k=0}^{p-1}\int\limits_{B}u(e^{\frac{k\pi i}{p}}y)Z^p_m(x,e^{\frac{k\pi i}{p}}y)dy.$$
Now, let $u$ be a polynomial of degree $M$, then $u$ is the sum of the homogeneous polynomials, so by Proposition \ref{S6} and the last equation we have the desired formula.
\end{proof}
\begin{Stw}
\label{S8}
The space of polyharmonic polynomials is a dense subset of  $b_p^2(\widehat{B}_p)$. 
\end{Stw}

\begin{proof}
The proof is the same as in the harmonic case (see {\cite[Lemma 8.8]{A-B-R}}). Here we use Proposition \ref{S2}.
\end{proof}

\begin{Le}[{\cite[Theorem 4]{G-M-2}}]
\label{L8}
Let $\zeta\in \widehat{S}_p$. Then there exists a constant $C>0$ such that 
$$|Z^p_m(x,\zeta)|_{\CC}\leq C pm^{n-2} ||x||^m$$
for every $x\in\widehat{B}_p$.
\end{Le}

\begin{Tw}
\label{T1}
The polyharmonic Bergman kernel is given by
$$R_p(x,y)=\frac{1}{n\Omega_n}\sum_{m=0}^{\infty}(n+2m)Z^p_m(x,y),$$
where the series converges absolutely and uniformly on $K\times \widehat{B}_p$ for every compact  subset $K\subset \widehat{B}_p$.
\end{Tw}
\begin{proof}
By (\ref{e10}) and Proposition \ref{S7} and Proposition \ref{S8},  we need only to show the convergence of the series.
By Lemma \ref{L8}  there is 
$$|Z^p_m(x,y)|_{\CC}=||\overline{y}|^m|_{\CC}|Z^p_m(x,y/|y|)|_{\CC}\leq Cpm^{n-2}||x||^m ||\overline{y}||^m\leq Cpm^{n-2}||x||^m,$$
thus
$$\max_{(x,y)\in K\times \widehat{B}_p}\sum_{m=0}^{\infty}(n+2m)|Z^p_m(x,y)|_{\CC}\leq Cp\max_{(x,y)\in K\times \widehat{B}_p}
\sum_{m=0}^{\infty}(n+2m)m^{n-2}||x||^m,
$$
but 
$$\frac{(n+2(m+1))(m+1)^{n-2}}{(n+2m)m^{n-2}} \rightarrow 1 \ \  \textrm{as} \ \ m\rightarrow \infty,$$
what completes the proof.
\end{proof}

\begin{Tw}
\label{T2}
The polyharmonic Bergman kernel is given by:
\begin{eqnarray*}
R_p(x,y) &=& \frac{1}{n\Omega_n}\left( nP_p(x,y)+\frac{d}{dt}P_p(tx,ty)\biggr|_{t=1}\right)\\
&=&\frac{(n-4p)|x|^{2p+2}|\overline{y}|^{2p+2}+(8px\overline{y}-n-4p)|x|^{2p}|\overline{y}|^{2p}+n(1-|x|^{2}|\overline{y}|^{2})}{n\Omega_n(1-2x\overline{y}+|x|^2|\overline{y}|^2)^{n/2+1}}.
\end{eqnarray*}
\end{Tw}
\begin{proof}
The proof is almost the same as in the harmonic case. By homogeneity of the zonal polyharmonics we have
$$2mZ^p_m(x,y)=\frac{d}{dt}t^{2m} Z^p_m(x,y)\biggr|_{t=1}=\frac{d}{dt}Z^p_m(tx,ty)\biggr|_{t=1}.$$
For the second equality by (\ref{e8}) we compute
\begin{multline*}
\frac{d}{dt}P_p(tx,ty)\biggr|_{t=1}=\frac{d}{dt}\frac{1-t^{4p}|x|^{2p}|\overline{y}|^{2p}}{(1-2t^2x\overline{y}+t^4|x|^2|\overline{y}|^2)^{n/2}}\biggr|_{t=1}\\
=\frac{(-4pt^{4p-1}|x|^{2p}|\overline{y}|^{2p})(t^4|x|^{2}|\overline{y}|^{2}-2t^2x\overline{y}+1)^{n/2}}{(1-2t^2x\overline{y}+t^4|x|^2|\overline{y}|^2)^{n}}\biggr|_{t=1}\\
-\frac{\frac{1}{2}n(1-t^{4p}|x|^{2p}|\overline{y}|^{2p})(-4tx\overline{y}+4t^3|x|^{2}|\overline{y}|^{2})(1-2t^2x\overline{y}+t^4|x|^2|\overline{y}|^2)^{n/2-1}}{(1-2t^2x\overline{y}+t^4|x|^2|\overline{y}|^2)^{n}}\biggr|_{t=1}\\
=\frac{2n(1-|x|^{2p}|\overline{y}|^{2p})(x\overline{y}-|x|^{2}|\overline{y}|^{2})-4p|x|^{2p}|\overline{y}|^{2p}(|x|^{2}|\overline{y}|^{2}-2x\overline{y}+1)}{(1-2x\overline{y}+|x|^2|\overline{y}|^2)^{n/2+1}}
\end{multline*}
and
$$nP_p(x,y)=\frac{n-n|x|^{2p}|\overline{y}|^{2p}-2nx\overline{y}+2nx\overline{y}|x|^{2p}|\overline{y}|^{2p}+n|x|^{2}|\overline{y}|^{2}-n|x|^{2p+2}|\overline{y}|^{2p+2}}{(1-2x\overline{y}+|x|^2|\overline{y}|^2)^{n/2+1}}.$$
Adding the last equalities we get the desired formula.
\end{proof}
In the next theorem we express the polyharmonic Bergman kernel in the terms of the harmonic Bergman kernel and harmonic Poisson kernel.
\begin{Tw}
The polyharmonic Bergman kernel is given by
\label{T3}
\begin{equation}
R_p(x,y)=\frac{1-|x|^{2p}|\overline{y}|^{2p}}{1-|x|^2|\overline{y}|^{2}} R(x,y)+\frac{1}{n\Omega_n}\sum_{k=0}^{p-1}4k|x|^{2k}|\overline{y}|^{2k}P(x,y).
\end{equation}
\end{Tw}
\begin{proof}
By Theorem \ref{T1} and (\ref{e3}) we have
\begin{eqnarray*}
R_p(x,y)&=&\frac{1}{n\Omega_n}\sum_{m=0}^{\infty}(n+2m)Z^p_m(x,y)\\
&=&\frac{1}{n\Omega_n}\sum_{k=0}^{p-1}|x|^{2k}|\overline{y}|^{2k}\sum_{m=-2k}^{\infty}(n+2m+4k)Z_{m}(x,y)\\
&=&\frac{1}{n\Omega_n}\sum_{k=0}^{p-1}|x|^{2k}|\overline{y}|^{2k}\sum_{m=0}^{\infty}(n+2m+4k)Z_{m}(x,y)\\
&=&\frac{1}{n\Omega_n}\sum_{k=0}^{p-1}|x|^{2k}|\overline{y}|^{2k}\sum_{m=0}^{\infty}(n+2m)Z_{m}(x,y)\\
&&+\frac{1}{n\Omega_n}\sum_{k=0}^{p-1}4k|x|^{2k}|\overline{y}|^{2k}\sum_{m=0}^{\infty} Z_{m}(x,y).
\end{eqnarray*}
Using  (\ref{e8}) and Theorem \ref{T1} with $p=1$ we obtain
\begin{eqnarray*}
R_p(x,y)&=&\sum_{k=0}^{p-1}|x|^{2k}|\overline{y}|^{2k}R(x,y)+\frac{1}{n\Omega_n}\sum_{k=0}^{p-1}4k|x|^{2k}|\overline{y}|^{2k}\sum_{m=0}^{\infty} Z_{m}(x,y)\\
&=&\frac{1-|x|^{2p}|\overline{y}|^{2p}}{1-|x|^2|\overline{y}|^{2}} R(x,y)+\frac{1}{n\Omega_n}\sum_{k=0}^{p-1}4k|x|^{2k}|\overline{y}|^{2k}P(x,y).
\end{eqnarray*}
\end{proof}

\section{The weighted Bergman kernel}
Let $\alpha+n>0$ and $\beta>-1$. Let us consider the set of polyharmonic functions of order $p$ on the ball $B$ such that
\begin{equation}
\label{eq:weighted}
||u||_{b_{p,\alpha,\beta}^2}:=\left(\frac{1}{p}\sum_{k=0}^{p-1}\int\limits_B\left|u(e^\frac{k\pi i}{p}y) \right|_\CC^2|y|^\alpha(1-|y|^{2})^\beta dy \right)^{1/2}
<\infty.
\end{equation}
This space is called a polyharmonic weighted Bergman space with weights $\alpha,\beta$ and we denote it by $b_{p,\alpha,\beta}^2(\widehat{B}_p)$. Hence
$$b_{p,\alpha,\beta}^2(\widehat{B}_p):=\mathcal{A}_{\Delta^p}(B)\cap L^2(\widehat{B}_p,|y|^\alpha(1-|y|^{2})^\beta dy),$$
where $L^2(\widehat{B}_p,|y|^\alpha(1-|y|^{2})^\beta dy)$ is the space of measurable functions on $\widehat{B}_p$ which satisfy (\ref{eq:weighted}).

By mean value property for polyharmonic functions (Lemma \ref{L7}) we conclude that
\begin{Le}[{\cite[Introduction]{T}}]
\label{L9}
Let $n+\alpha>0,\beta>-1$, then for every compact subset $K\subset B$ and $x\in K$, there exists a constant $C=C(K,n,p)$ such that
$$|u(x)|_\CC^2\leq C  \int\limits_B\left|u(y) \right|_\CC^2|y|^\alpha(1-|y|^{2})^\beta dy$$
for every $u\in \mathcal{A}_{\Delta^p}(B)\cap L^2(B,|y|^\alpha(1-|y|^{2})^\beta dy).$
\end{Le}

\begin{Stw}
\label{S9}
Let $n+\alpha>0,\beta>-1$, then for every compact subset $K\subset B$ and $x\in K$, there exists a constant $C=C(K,n,p)$ such that
$$|u(x)|_\CC\leq C  ||u||_{b_{p,\alpha,\beta}^2}$$
for every $u\in b_{p,\alpha,\beta}^2(\widehat{B}_p).$
\end{Stw}
\begin{proof}
The proof follows from Lemma \ref{L9} and it is similar to the proof of Proposition \ref{S3}.
\end{proof}

\begin{Wn}
\label{W1}
The space $b_{p,\alpha,\beta}^2(\widehat{B}_p)$ is a closed subspace of the Hilbert space $L^2(\widehat{B}_p, |y|^\alpha(1-|y|^2)^\beta dy)$ with the inner product
\begin{equation}
\label{e12}
\left\langle u,v \right\rangle_{b_{p,\alpha,\beta}^2}=
\frac{1}{p}\sum_{k=0}^{p-1}\int\limits_B 
 u(e^\frac{k\pi i}{p}y) \overline{v(e^\frac{k\pi i}{p}y)} |y|^\alpha(1-|y|^{2})^\beta dy.
\end{equation}
\end{Wn}

By the last corollary we conclude that $b_{p,\alpha,\beta}^2(\widehat{B}_p)$ 
is a Hilbert space with the  inner product (\ref{e12}). Again, let $x\in \widehat{B}_p$ be a fixed point and
let the linear functional $\Lambda:b_{p,\alpha,\beta}^2(\widehat{B}_p)\rightarrow \CC$ be such that $\Lambda_x(u)=u(x)$, 
then by Corollary \ref{W1} and Riesz Theorem there exists the function $R_{p,\alpha,\beta}\in b_{p,\alpha,\beta}^2(\widehat{B}_p)$ such that for every $u\in b_{p,\alpha,\beta}^2(\widehat{B}_p)$ we have
\begin{eqnarray*}
u(x)&=&\left\langle u,R_{p,\alpha,\beta}(x,\cdot) \right\rangle_{b_{p,\alpha,\beta}^2}\\
&=& \frac{1}{p}\sum_{k=0}^{p-1}\int\limits_B 
 u(e^\frac{k\pi i}{p}y) \overline{R_{p,\alpha,\beta}(x,e^\frac{k\pi i}{p}y)} |y|^\alpha(1-|y|^{2})^\beta dy.
\end{eqnarray*}
The function $R_{p,\alpha,\beta}(x,\cdot)$ is called \emph{a polyharmonic weighted Bergman kernel}. Let's note that by the Lemma \ref{L1}, the harmonic weighted Bergman kernel $R_{1,\alpha,\beta}(x,y)$ can be extended from $B\times B$ on $\widehat{B}_p\times \widehat{B}_p.$

\begin{Uw}
\label{U1}
It is easy to observe that the analogous  properties given in Proposition \ref{S5}, Proposition \ref{S6} and Proposition \ref{S8} from the previous section hold, one can change $R_p$  to $R_{p,\alpha,\beta}$. 
\end{Uw}

\begin{Stw}
\label{S10}
Let $u$ be a polynomial of degree $M$. Then
\begin{multline*}
u(x)=\sum_{m=0}^{M}\frac{2\Gamma(m+\frac{n+\alpha}{2}+\beta+1)}{pn\Omega_n \Gamma(\beta+1)\Gamma(m+\frac{\alpha+n}{2})}\\
\times \sum_{k=0}^{p-1}\int\limits_{B}u(e^{\frac{k\pi i}{p}}y)Z^p_m(x,e^{\frac{k\pi i}{p}}y) |y|^\alpha(1-|y|^{2})^\beta dy.
\end{multline*}
\end{Stw}
\begin{proof}
The proof is similar to the proof of Proposition \ref{S7}.  First we assume that $u\in\mathcal{H}_m^p(\CC^n)$. Then
\begin{multline*}
\int\limits_{B}u(y)Z^p_m(x,y)|y|^\alpha(1-|y|^{2})^\beta dy\\
=n\Omega_n\int\limits_{0}^1r^{n+2m+\alpha-1}(1-r^2)^\beta\int\limits_S u(\zeta)Z^p_m(x,\zeta)d\zeta dr\\
=n\Omega_n u(x)\int\limits_0^1 r^{n+2m+\alpha-1}(1-r^2)^\beta dr\\
=n\Omega_n \frac{\Gamma(\beta+1)\Gamma(m+\frac{n+\alpha}{2})}{2\Gamma(m+\frac{n+\alpha}{2}+\beta+1)} u(x).
\end{multline*}
Hence
$$u(x)=\frac{2\Gamma(m+\frac{n+\alpha}{2}+\beta+1)}{n\Omega_n \Gamma(\beta+1)\Gamma(m+\frac{n+\alpha}{2})}\int\limits_{B}u(y)Z^p_m(x,y)|y|^\alpha(1-|y|^{2})^\beta dy.$$
By homogeneity (see Remark \ref{U1}) we get 
\begin{multline*}
u(x)=\frac{2\Gamma(m+\frac{n+\alpha}{2}+\beta+1)}{pn\Omega_n \Gamma(\beta+1)\Gamma(m+\frac{n+\alpha}{2})}\\
 \times \sum_{k=0}^{p-1}\int\limits_{B}u(e^{\frac{k\pi i}{p}}y)Z^p_m(x,e^{\frac{k\pi i}{p}}y) |y|^\alpha(1-|y|^{2})^\beta dy.
\end{multline*}
Now let $u$ be a polynomial of degree $M$, then by last equation and  Remark \ref{U1} we obtain the desired formula.
\end{proof}
\begin{Tw}
\label{T4}
The polyharmonic Bergman kernel is given by
$$R_{p,\alpha,\beta}(x,y)=\frac{1}{n\Omega_n}\sum_{m=0}^{\infty}\frac{2\Gamma(m+\frac{n+\alpha}{2}+\beta+1)}{ \Gamma(\beta+1)\Gamma(m+\frac{n+\alpha}{2})}  Z^p_m(x,y),$$
where the series converges absolutly and uniformly on $K\times \widehat{B}_p$ for every compact  subset $K\subset \widehat{B}_p$.
\end{Tw} 
\begin{proof}
By Remark \ref{U1} and Proposition \ref{S10}, we need to show the convergence. As in the proof of Theorem \ref{T1} we have
\begin{multline*}
\max_{(x,y)\in K\times \widehat{B}_p}\sum_{m=0}^{\infty}\frac{2\Gamma(m+\frac{n+\alpha}{2}+\beta+1)}{ \Gamma(\beta+1)\Gamma(m+\frac{n+\alpha}{2})}|Z^p_m(x,y)|_{\CC}\\
\leq Cp\max_{(x,y)\in K\times \widehat{B}_p}
\sum_{m=0}^{\infty}\frac{2\Gamma(m+\frac{n+\alpha}{2}+\beta+1)}{ \Gamma(\beta+1)\Gamma(m+\frac{n+\alpha}{2})}m^{n-2}||x||^m.
\end{multline*}
Moreover
\begin{multline*}
\frac{2\Gamma(m+\frac{n+\alpha}{2}+\beta +2)}{ \Gamma(\beta+1)\Gamma(m+\frac{n+\alpha}{2}+1)}\cdot \frac{\Gamma(\beta+1)\Gamma(m+\frac{n+\alpha}{2})}{2\Gamma(m+\frac{n+\alpha}{2}+\beta+1)}\\
=\frac{n+2m+\alpha+2\beta+2}{n+2m+\alpha}  \rightarrow 1 \ \ \textrm{as} \ \ m\rightarrow \infty
\end{multline*}
and this completes the proof.
\end{proof} 
We may also give the counterpart of Theorem \ref{T2} using the fractional derivatives in the Riemann-Liouville sense (see for example \cite{P3}). Let's recall the definitions. 

Let $l>0$, then the primitive of $u\in L^1(0,1)$ is as follows
$$D^{-l}u(t)=\frac{1}{\Gamma(l)}\int\limits_0^1 \frac{u(\tau)}{(t-\tau)^{1-l}}d\tau.$$
The derivative of order $l$ is as follows
$$D^lu(t)=\frac{d^j}{dt^j}\left(D^{-(j-l)}u(t) \right),$$
where $j$ is an integer number such that $j-1\leq l \leq j$. As in harmonic case (see  \cite{P3}), using the identity
$$D^{l+1}t^k=\frac{\Gamma(k+1)}{\Gamma(k-l)}t^{k-l-1}$$
and again the formula (\ref{e8}), we obtain the following theorem for polyharmonic case:
\begin{Tw}
\label{T5}
The polyharmonic weighted Bergman kernel is given by
$$R_{p,\alpha,\beta}(x,y)=\frac{2}{n\Gamma(\beta+1)\Omega_n}D^{\beta+1}\left(t^{\frac{n+\alpha}{2}+\beta }P_p(tx,y)\right)\biggr|_{t=1}.$$
\end{Tw}
\begin{Tw}
\label{T6}
The polyharmonic weighted Bergman kernel is given by
$$R_{p,\alpha,\beta}(x,y)=\sum_{k=0}^{p-1}|x|^{2k}|\overline{y}|^{2k}R_{1,\alpha+4k,\beta}(x,y),$$
in particular
$$R_{p}(x,y)=\sum_{k=0}^{p-1}|x|^{2k}|\overline{y}|^{2k}R_{1,4k,0}(x,y).$$
\end{Tw}
\begin{proof}
By Theorem \ref{T4} we have
$$R_{p,\alpha,\beta}(x,y)=\frac{1}{n\Omega_n}\sum_{m=0}^{\infty}\frac{2\Gamma(m+\frac{n+\alpha}{2}+\beta+1)}{ \Gamma(\beta+1)\Gamma(m+\frac{n+\alpha}{2})}  Z^p_m(x,y).$$
From (\ref{e3}) we get
\begin{eqnarray*}
R_{p,\alpha,\beta}(x,y)&=&\frac{1}{n\Omega_n}\sum_{k=0}^{p-1}\sum_{m=0}^{\infty}\frac{2\Gamma(m+\frac{n+\alpha}{2}+\beta+1)}{ \Gamma(\beta+1)\Gamma(m+\frac{n+\alpha}{2})}  |x|^{2k}|\overline{y}|^{2k}Z_{m-2k}(x,y)\\
&=&\frac{1}{n\Omega_n}\sum_{k=0}^{p-1}\sum_{m=0}^{\infty}\frac{2\Gamma(m+2k+\frac{n+\alpha}{2}+\beta+1)}{ \Gamma(\beta+1)\Gamma(m+2k+\frac{n+\alpha}{2})} |x|^{2k}|\overline{y}|^{2k} Z_{m}(x,y)\\
&=&\frac{1}{n\Omega_n}\sum_{k=0}^{p-1}\sum_{m=0}^{\infty}\frac{2\Gamma(m+\frac{n+\alpha+4k}{2}+\beta+1)}{ \Gamma(\beta+1)\Gamma(m+\frac{n+\alpha+4k}{2})} |x|^{2k}|\overline{y}|^{2k} Z_{m}(x,y).
\end{eqnarray*}
Using Theorem \ref{T4} we obtain the desired formula.
\end{proof}

\section*{Acknowledgements}
The author is grateful to S{\l}awomir Michalik for many valuable comments.

\end{document}